\documentclass[11pt,reqno]{amsart}

\usepackage[T1]{fontenc}
\usepackage[utf8]{inputenc}
\usepackage[american]{babel}

\usepackage[square,numbers,sort&compress]{natbib}
\usepackage{subfig}
\usepackage{hyperref}
\usepackage{amsmath}
\usepackage{amssymb}
\usepackage{amsthm}
\usepackage{bm}

\newtheorem{lemma}{Lemma}[section]
\newtheorem{proposition}[lemma]{Proposition}
\newtheorem{theorem}[lemma]{Theorem}

\newtheorem{remark}[lemma]{Remark}
\newtheorem{hypothesis}[lemma]{Hypothesis}

\usepackage{tikz}
\usetikzlibrary
  {arrows,calc,through,backgrounds,matrix,positioning,decorations.pathmorphing}

\newcommand{\geps}{\gamma_{\varepsilon}}

\newcommand{\gsig}{\gamma_{\sigma}}
\newcommand{\gj}{\gamma_J}

\newcommand{\R}{\mathbb{R}}

\newcommand{\curl}{\mathrm{curl}}

\renewcommand{\Re}{\mathrm {Re}\,}
\renewcommand{\Im}{\mathrm {Im}\,}

\usepackage{mathtools}

\newcommand{\supp}{\rm{supp}\,}
\newcommand{\cl}{\nabla\times}
\newcommand{\Ltc}{(L^2(\Omega))^3}
\newcommand{\Ltcb}{(L^2(\partial\Omega))^3}

\newcommand{\Hc}{H(\curl,\Omega) }
\newcommand{\Hmhdiv}{H^{-1/2}(\text{div},\partial\Omega)}
\newcommand{\Hmhcl}{H^{-1/2}(\text{curl},\partial\Omega)}
\newcommand{\Hmh}{H^{-1/2}(\partial\Omega)}

\newcommand{\vA}{\mathbf{A}}
\newcommand{\vB}{\mathbf{B}}
\newcommand{\vC}{\mathbf{C}}
\newcommand{\vE}{\mathbf{E}}
\newcommand{\vH}{\mathbf{H}}
\newcommand{\vF}{\mathbf{F}}
\newcommand{\vG}{\mathbf{G}}
\newcommand{\vJ}{\mathbf{J}}
\newcommand{\vQ}{\mathbf{Q}}
\newcommand{\vx}{\mathbf{x}}
\newcommand{\vn}{\hat{\mathbf{n}}}
\newcommand{\vh}{\mathbf{h}}
\newcommand{\vg}{\mathbf{g}}
\newcommand{\vf}{\mathbf{f}}
\newcommand{\vk}{\mathbf{k}}
\newcommand{\vu}{\mathbf{u}}
\newcommand{\vv}{\mathbf{v}}
\newcommand{\vphi}{\boldsymbol{\phi}}

\newcommand{\vdg}{\mathbf{\delta}\mathbf{g}}

\newcommand{\eps}{\varepsilon}

\newcommand{\bn}{\mathbf n}

\newcommand{\bE}{\mathbf E} 
 
 \newcommand{\bJ}{\mathbf J}

\newcommand{\bQ}{\mathbf Q}

\title[Acousto-electromagnetic Inverse Source Problem]{Inverse Source Problem for Acoustically- Modulated Electromagnetic Waves}

\author{%
}
\author{Wei Li, John C. Schotland, Yang Yang, Yimin Zhong}
\address{Department of Mathematical Sciences, DePaul University, Chicago, IL 60604}
\email{wei.li@depaul.edu}
\address{Department of Mathematics and Department of Physics, Yale University, New Haven, CT 06511}
\email{john.schotland@yale.edu}
\address{Department of Computational Mathematics, Science and Engineering, Michigan State University, East Lansing, MI 48824}
\email{yangy5@msu.edu}
\address{Department of Mathematics, Duke University, Durham, NC 27708}
\email{yimin.zhong@duke.edu}

\begin{document}

\maketitle

\begin{abstract}
We propose a method to reconstruct the electrical current density from acoustically-modulated boundary measurements of time-harmonic electromagnetic fields. We show that the current can be uniquely reconstructed with Lipschitz stability. We also report numerical simulations to illustrate the analytical results.
\end{abstract} 	

\section{Introduction}
The inverse source problem for the Maxwell equations is of fundamental interest and considerable practical importance, with applications ranging from geophysics to biomedical imaging~\cite{Isakov_book,Albanese_2006,Bao_2020,Bleistein_1977,zhao2018inverse}. The problem is usually stated in the following form: determine the electric current density from boundary measurements of the electric and magnetic fields. It is well known that this problem is underdetermined and does not admit a unique solution, due to the existence of so-called nonradiating sources~\cite{Devaney_1973}.
However, if the source is spatially localized or some other a priori information is available, it is often possible to characterize the source to some extent. Such a method is applied to the localization of low-frequency electric and magnetic signals originating from current sources in the brain or heart~\cite{Plonsey_Barr}.

In this paper, we propose an alternative approach to the electromagnetic inverse source problem. In this approach, which is an extension of the authors' previous work on the acousto-electric inverse source problem for static fields~\cite{Li_2021}, a wavefield is used to control the material properties of a medium of interest, which is then probed by a second wavefield. Also see related work on hybrid imaging~\cite{Ammari_2008,Bal_2016,Bal2010,BalGuoMonard2014,Bal_2010,Bal_2014,Capdeboscq_2009,Chung_2020_1,Chung_2020_2,Chung_2017,Chung2021,Gebauer_2008,Kuchment_2011,kuchment2012stabilizing,li2019hybrid,li2020inverse,Li_2021,Nachman,Schotland_2020,Triki_2010}. Here the electric current density as well as the conductivity, electric permittivity and magnetic permeability are spatially modulated by an acoustic wave. In this manner, we find that it is possible to uniquely recover the current density from boundary measurements of the fields with Lipschitz stability.

The remainder of this paper is organized as follows. In Section 2 we introduce a model for the acoustic modulation of the current density and the material parameters. 
In Sections 3 and 4 this model is used to formulate the inverse source problem and thereby derive an internal functional from which the source may be recovered.
Numerically simulated reconstructions are given in Section 5. Finally, our conclusions are presented in Section 6.

\section{Model}
We begin by developing a simple model for acoustic modulation of the electrical current density and material parameters, following the approach of~\cite{Li_2021}.
We begin by considering
the time-harmonic Maxwell equations in a bounded domain $\Omega\subset\R^3$:
\begin{equation}
\label{eq:frequency}
\begin{aligned}
i\omega \varepsilon  \,  \vE +  \cl \vH &= \vJ + \sigma \vE  & \quad\quad  \text{in } \quad \Omega,\\
-i\omega \mu \vH +  \cl \vE &= 0 & \quad\quad \text{in } \quad \Omega \ .
\end{aligned}
\end{equation}
We also impose the impedance boundary condition
\begin{equation}\label{eq:imp}
\vH \times \vn   - \lambda (\vn\times\vE)\times\vn=0  \quad\quad \text{on } \quad \partial\Omega \ ,
\end{equation}
which arises since $\Omega$ is taken to be enclosed by a good conductor.
Here the vector functions $\vJ$, $\vE$ and $\vH$ are the current density, the electric field, and the magnetic field, respectively. The scalar functions $\varepsilon$, $\mu$, $\sigma$ and $\lambda$ are the  {electric permittivity}, {magnetic permeability}, {conductivity}, and {surface impedance}, respectively. The vector $\vn$ is the outward unit outward normal to $\Omega$ and $\omega$ is a fixed frequency. Note that in the above, we do not write the equations governing the divergence of $\vE$ and $\vH$ which are not needed in what follows.

The inverse source problem is to reconstruct the source $\vJ$ from boundary measurements, assuming that the coefficients $\mu$, $\varepsilon$, $\sigma$, $\lambda$ are known. A typical measurement is the tangential electric field on the boundary:
\begin{equation}\label{eq:meas}
\vg:=(\vn\times\vE)\times\vn|_{\partial\Omega} .
\end{equation}
This problem does not have a unique solution~\cite{Devaney_1973}. That is, distinct sources may give rise to the same boundary measurements.

\begin{remark} \label{rem:equivofmeas}
An alternative measurement is $\vh:=\vn\times\vH|_{\partial\Omega}$. Knowledge of $\vg$ is equivalent to knowledge of $\vh$ when the impedance boundary condition~\eqref{eq:imp} is taken into account, since 
$ \vh = -\lambda \vg $ on $\partial\Omega$.
\end{remark}

We now examine the effect of acoustic modulation. Following~\cite{Li_2021,Bal_2010}, we consider a system of charge carriers in a fluid, in which a small-amplitude acoustic plane wave propagates. It follows that the current density $\vJ_{\delta}$ is modulated according to
\begin{align}
\vJ_{\delta} & = \vJ (1+\delta \cos(\vk\cdot \vx + \varphi)) ,
\end{align}
where $\vJ$ is the conductivity in the absence of the acoustic wave, $\delta \ll 1$ is a small parameter that is proportional to the acoustic pressure, $\vk$ is the wave vector of the acoustic wave and $\varphi$ is its phase. Likewise, the conductivity $\sigma_\delta$ and permittivity $\varepsilon_\delta$ are also modulated:
\begin{align*}
\varepsilon_{\delta} & = \varepsilon (1+\delta \geps  \cos(\vk\cdot \vx + \varphi)), \\
\sigma_{\delta} & = \sigma (1+\delta\gsig  \cos(\vk\cdot \vx + \varphi)), \\
\end{align*}
where $\sigma$ and $\varepsilon$ are the unmodulated conductivity and permittivity, and the constants $\geps, \gsig$ are known as the {elasto-electric constants}.
For simplicity we assume that the impedance $\lambda$ is not affected by the acoustic modulation. It follows that the modulated electric and magnetic fields
$\vE_{\delta}$ and $\vH_{\delta}$ satisfy the modified Maxwell equations
\begin{equation} \label{eq:pert}
    \begin{aligned}
i\omega \varepsilon_\delta  \,  \vE_\delta +  \cl \vH_\delta &= \vJ_\delta + \sigma_\delta \vE_\delta  & \quad\quad \text{in } \quad \Omega,\\
-i\omega \mu \vH_\delta +  \cl \vE_\delta &= 0 & \quad\quad \text{in } \quad \Omega,
\end{aligned}
\end{equation}
together with the boundary condition
\begin{equation}\label{eq:SilverE_pert}
\vH_\delta \times \vn   - \lambda (\vn\times\vE_\delta)\times\vn=0  \quad\quad \text{on }\quad \partial\Omega.
\end{equation}
The corresponding boundary measurement becomes
\begin{equation}\label{eq:UMmeas}
\vg_\delta:=(\vn\times\vE_\delta)\times\vn|_{\partial\Omega}.
\end{equation}

%

\bigskip
\section{Internal Functional}
In this section, we derive the internal functional from boundary measurements of the electric field. We also introduce the necessary function spaces and specify certain technical requirements on the conductivity and permittivity.

\subsection{Function Spaces}
We will use the following standard spaces to discuss the wellposedness of the Maxwell's equations~\cite{Cessenat1996}. Let $\Omega\subset\R^3$ be an open bounded set with a $C^{1,1}$ boundary, and
$$
\Hc = \left\{ \vu\in \Ltc : \cl \vu \in \Ltc \right\}.
$$
The norm on $\Hc$ is given by 
$$
\| u\|_{\Hc} =\big(\| \vu \|^2_{\Ltc} + \| \cl \vu \|^2_{\Ltc} \big)^{1/2} .
$$
The two tangential trace maps $ \Gamma_{\tau}$ and $\Pi_{\tau}$ have the following definitions
 $$
 \begin{aligned}
 \Gamma_{\tau}: \Hc &\rightarrow \Hmhdiv  , \\
 \vu &\mapsto \vn\times \vu|_{\partial\Omega}
 \end{aligned}
 $$ 
 and
 $$
 \begin{aligned}
 \Pi_{\tau}:  \Hc &\rightarrow \Hmhcl ,  \\
 \vu&\mapsto (\vn\times \vu)\times\vn|_{\partial\Omega},
 \end{aligned}
 $$
where
$$
\Hmhdiv = \left\{ \vu\in (\Hmh)^3 : \text{div}_{\partial\Omega}\vu \in \Hmh \right\},
$$
and
$$
\Hmhcl = \left\{ \vu\in (\Hmh)^3 : \text{curl}_{\partial\Omega}\vu \in \Hmh \right\}.
$$
Here $\text{div}_{\partial\Omega}$ is the surface divergence and $\text{curl}_{\partial\Omega}$ is the surface curl. The two spaces $\Hmhdiv$ and $\Hmhcl$ are dual to each other.
To handle the impedance boundary condition, we define the tangential trace of a vector field
\begin{equation}\label{eq:tangtrace}
\vphi_T:= \Pi_{\tau}(\vphi)= (\vn\times \vphi)\times\vn|_{\partial\Omega},
\end{equation}
and the space
$$
X=\left\{ \vu\in\Hc: u_T\in \Ltcb \right\}.
$$
The norm on $X$ is 
$$
\|\vu\|_X^2=\|\vu\|_{\Hc}^2+\|\vu_T\|_{\Ltcb}^2.
$$
We denote the $\Ltc$-inner product by
$$
(\vu,\vv)_{\Ltc}:=\int_\Omega \vu\cdot\vv^* d {\bf x}, \quad \vu,\vv\in\Ltc,
$$
and the $\Ltcb$ inner product by
$$
\langle\vu,\vv\rangle_{\Ltcb} :=\int_{\partial\Omega} \vu\cdot\vv^* d {\bf x}, \quad \vu,\vv\in\Ltcb ,
$$
where $*$ denotes the complex conjugate.
We denote the dual paring of $\vu\in\Hmhdiv$ and $\vv\in\Hmhcl$ by
$\langle\vu,\vv\rangle$.

\medskip
\subsection{Assumptions and Weak Formulation}
We will make the following assumptions throughout this paper.
\begin{enumerate}
\item[A-1.]
The domain $\Omega$ is an open bounded connected domain in~$\mathbb R^3$ with $C^{1,1}$ boundary.
\item[A-2.]
The medium is nonmagnetic with $\mu=\mu_0$ in $\Omega$, where $\mu_0$ is the magnetic permeability in vacuum. The coefficients $\varepsilon$ and $\sigma$ are real piecewise $H^3(\Omega)$ functions.
\item[A-3.]
There exists positive constants $K_1$ and $K_2$, such that
\begin{equation}
K_1>\varepsilon, \lambda >K_2>0,\quad K_1>\sigma\geq0,
\end{equation} 
and the conductivity $\sigma$ is nonzero.
\item[A-4.]
The source $\vJ$ is an $\Ltc$ vector field and is compactly supported in~$\Omega$.
\end{enumerate}

\textbf{Remark:} We conclude from A-2 that $\varepsilon$ and $\sigma$ are piecewise $C^1$ by the Sobolev embedding theorem. We conclude from A-3 that $K_1 > \varepsilon_\delta > K_2 >0$ and $\sigma_\delta\geq 0$, so long as $\delta$ is sufficiently small.

\medskip
The modulated Maxwell equations~\eqref{eq:pert} and the impedance boundary condition~\eqref{eq:SilverE_pert} can be written in terms of only the electric field:
\begin{equation}\label{eq:forward}
\begin{aligned}
\cl \cl \vE_{\delta} - \mu(\omega^2\varepsilon_{\delta}+ i\omega\sigma_{\delta} )\vE_{\delta} = i\omega\mu \vJ_{\delta} \quad \text{in } \quad \Omega ,
\end{aligned}
\end{equation}
which is subject to the impedance boundary condition
\begin{equation}\label{eq:impE}
\left(\frac{1}{\mu} \cl \vE_\delta\right)\times \vn - i \omega\lambda (\vn\times\vE_\delta)\times\vn=0  \quad \text{on } \quad \partial\Omega .
\end{equation}

We say $\vE_\delta \in X$ is a {weak solution} of~\eqref{eq:forward} obeying the impedance boundary condition~\eqref{eq:impE} if for all $\vphi\in X$,
\begin{align}
\label{eq:weak}
\nonumber
&\left(\frac{1}{\mu} \cl \vE_\delta, \cl\vphi \right)_{\Ltc}  - \left((\omega^2\varepsilon+ i\omega\sigma)\vE_\delta, \vphi\right)_{\Ltc} - i\omega\langle \lambda \vE_{\delta T}, \vphi_T\rangle \\
&= i\omega \left(\vJ, \vphi\right)_{\Ltc}.
\end{align}
It follows from Assumptions A1-A4, that the weak solution $\vE_\delta\in X$ exists and is unique~\cite{Monk2003}. 

\subsection{Internal Functional}
We now derive the internal functional for both classical and weak solutions.
To proceed, we consider the fields $\vF$ and $\vG$ which obey the  Maxwell equations without sources:
\begin{equation}\label{eq:conj1}
\begin{aligned}
i\omega \varepsilon  \,  \vF^* +  \cl \vG^* &=  \sigma \vF^* \quad \text{in } \quad \Omega,\\
-i\omega \mu \vG^* +  \cl \vF^* &= 0  \quad \text{in } \quad \Omega,
\end{aligned}
\end{equation}
along with the impedance boundary condition
$$
\vG^* \times \vn   - \lambda (\vn\times\vF^*)\times\vn = \mathfrak{g}  \quad \text{on } \quad\partial\Omega,
$$
where $\mathfrak{g}\in \Ltcb$. 
Equivalently, 
\begin{equation}\label{eq:conj2}
\begin{aligned}
& \cl \cl \vF^* - \mu(\omega^2\varepsilon + i\omega\sigma) \vF^*  = 0 \quad \text{in } \quad \Omega\\
&\left(\frac{1}{i{\omega}\mu} \nabla\times \vF^*\right) \times \vn   - \lambda (\vn\times\vF^*)\times\vn = \mathfrak{g}  \quad \text{on } \quad\partial\Omega.
\end{aligned}
\end{equation}
Note that (\ref{eq:conj1})  are explicitly solvable, since the required coefficients are known.
Next, we take the inner product of~\eqref{eq:pert} with $\vF^*$, the inner product of~\eqref{eq:conj2} with $\vE_\delta$, and then subtract to obtain
\begin{align*}
    \nabla \times  \nabla \times \vE_\delta \cdot \vF^* - \nabla \times  \nabla \times \vF^* \cdot \vE_\delta \\
    = \mu[\omega^2 (\varepsilon_\delta-\varepsilon) + i \omega (\sigma_\delta - \sigma) ] \vF^* \cdot \vE_\delta + i\mu\omega \vJ_\delta \cdot \vF^*.
\end{align*}
Integrating the above result over $\Omega$ and using the vector identity $(\nabla\times \vA)\cdot \vB = \nabla\cdot \left(\vA\times \vB\right) + (\nabla\times \vB)\cdot \vA$, we find that
\begin{align*}
    & \int_\Omega \left[\nabla\cdot \left(\frac{1}{\mu} \nabla \times \vE_\delta \times \vF^*\right) + \left(\nabla \times \vF^*\right)\cdot \left(\frac{1}{\mu} \nabla \times \vE_\delta\right)\right] d {\bf x} \\
    & - \int_\Omega \left[\nabla\cdot \left(\frac{1}{\mu} \nabla \times \vF^* \times \vE_\delta\right) + \left(\nabla \times \vE_\delta\right)\cdot \left(\frac{1}{\mu} \nabla \times \vF^*\right)\right]d {\bf x} \\
    = &\int_\Omega [\omega^2 (\varepsilon_\delta-\varepsilon) + i \omega (\sigma_\delta - \sigma) ] \vF^* \cdot \vE_\delta d {\bf x} + i\omega \vJ_\delta \cdot \vF^*
\end{align*}
We now integrate by parts the divergence terms, which using the relations $\frac{1}{\mu} \nabla \times \vE_\delta = i\omega \vH_\delta$ and $\frac{1}{\mu} \nabla \times \vF^* = i\omega \vG^*$ yields
\begin{align} 
   & i\omega \int_{\partial\Omega} \vn \cdot (  \vH_\delta \times \vF^*)d {\bf x} - i\omega \int_{\partial\Omega} \vn \cdot (\vG^* \times \vE_\delta) d {\bf x} \nonumber \\
    &= \int_\Omega [\omega^2 (\varepsilon_\delta-\varepsilon) + i \omega (\sigma_\delta - \sigma) ] \vF^* \cdot \vE_\delta d {\bf x} + i\omega \vJ_\delta \cdot \vF^*. \label{eq:diff}
\end{align}
Note that the boundary integral only depends on the tangential components of the fields $\vH_{\delta}$, $\vE_{\delta}$, $\vF$ and $\vG$, which are known from the boundary measurements~\eqref{eq:UMmeas}. Therefore, the left-hand side can be determined from experiment. For the right-hand side, we consider the asymptotic expansion in the small quantity $\delta$. The $O(1)$ term is
$$
i\omega \int_\Omega \vJ_\delta \cdot \vF^*.
$$
The  $O(\delta$) term is of the form
\begin{equation} \label{eq:linearterm}
\int_\Omega 
\left[ \left(\omega^2 \varepsilon \gamma_\varepsilon + i\omega \sigma \gamma_\sigma \right) \vF^*\cdot \vE + i\omega \gamma_J \vJ \cdot \vF^* \right] \cos(\vk\cdot \vx + \varphi))d {\bf x}.
\end{equation}
Varying $\vk$ and $\varphi$ in~\eqref{eq:linearterm}, and performing the inverse Fourier transform, we obtain the 
internal functional
\begin{equation}\label{eq:intscalar}
Q:=\left(\omega^2 \varepsilon \gamma_\varepsilon + i\omega \sigma \gamma_\sigma \right) \vF^*\cdot \vE + i\omega \gamma_J \vJ \cdot \vF^* ,
\end{equation}
which is known at every point in $\Omega$.

We make the following hypothesis to extract more information from the 
internal function \eqref{eq:intscalar}:
\begin{hypothesis}\label{hypo:indep}
There exists a finite open cover $\left\{\Omega_{\alpha}\right\}_{\alpha\in\Lambda}$ of $\Omega$, such that for each $\alpha\in\Lambda$, there exist three solutions to \eqref{eq:conj2} in $\Omega$, denoted $\vF_{1\alpha}^*$,  $\vF_{2\alpha}^*$ and $\vF_{3\alpha}^*$, that are linearly independent on $X_{\alpha}$.
\end{hypothesis}

The hypothesis means that, in each $\Omega_\alpha$, we can form the non-singular matrix $[\vF^*_{1\alpha},\vF^*_{2\alpha},\vF^*_{3\alpha}]$, where $\vF^*_{j\alpha}$ is the $j$th column, $j=1,2,3$. Let $Q_{j\alpha}$ be the internal functional defined as in~\eqref{eq:intscalar}, with $\vF^*$ replaced by $\vF_{j\alpha}^*$. Given the row vector $[Q_{1\alpha}, Q_{2\alpha}, Q_{3\alpha}]$, we have 
$$
[Q_{1\alpha}, Q_{2\alpha}, Q_{3\alpha}]^T = [\vF_{1\alpha}^*, \vF_{2\alpha}^*, \vF_{3\alpha}^*]^T \left[ \left(\omega^2 \varepsilon \gamma_\varepsilon + i\omega \sigma \gamma_\sigma \right) \vE + i\omega \gamma_J \vJ \right], \quad \text{ in } \Omega_\alpha
$$
where we view $\vE$ and $\vJ$ as column vectors, and $T$ denotes the transpose.
Therefore, if we define $\vQ \in (L^2(\Omega))^3$ by specifying its restrictions according to
$$
\vQ|_{\Omega_\alpha} := [\vF_{1\alpha}^*, \vF_{2\alpha}^*, \vF_{3\alpha}^*]^{-T}
[Q_{1\alpha}, Q_{2\alpha}, Q_{3\alpha}]^T,
$$
then $\vQ$ is well-defined since both $\vE$ and $\vJ$ are global vector fields over $\Omega$, and we have
\begin{equation} \label{eq:intvector}
\vQ = 
 i\omega \gamma_J \vJ + \left(\omega^2 \varepsilon \gamma_\varepsilon + i\omega \sigma \gamma_\sigma \right) \vE.
\end{equation}
Note that we view $\vQ$ as a vector-valued internal functional.

\section{Inverse Problem and Internal Functional}

It follows from the above discussion that the inverse problem consists of recovering the source current  $\vJ$ from the internal functional $\vQ$.  In this section we will derive a reconstruction procedure that uniquely recovers $\vJ$ with Lipschitz stability. The analysis depends critically on whether the constant $\gamma_J$ vanishes. 

\subsection{Case I: $\gamma_J=0$.}
In this situation, the equality~\eqref{eq:intvector} does not involve $\vJ$ directly. 

\begin{proposition} \label{thm:case1}
Suppose the assumptions A1-A4 and the hypothesis~\eqref{hypo:indep} hold. If $\gamma_J=0$, then we have the following two subcases:
\begin{enumerate}
    \item[(I.1)] If $\Omega\subseteq \supp(\omega^2 \varepsilon \gamma_\varepsilon + i\omega \sigma \gamma_\sigma)$, then the source $\vJ$ is uniquely determined with the stability estimate
$$\|\vJ - \tilde{\vJ}\|_{X^*} \leq C \Big\|\frac{\vQ - \tilde{\vQ}}{\omega^2 \varepsilon \gamma_\varepsilon + i\omega \sigma \gamma_\sigma}\Big\|_X ,
$$
for some constant $C>0$ independent of $\vJ, \tilde{\vJ}$.
    \item[(I.2)] If $\Omega\not\subseteq \supp(\omega^2 \varepsilon \gamma_\varepsilon + i\omega \sigma \gamma_\sigma)$, then the source $\vJ$ cannot be uniquely determined.
\end{enumerate}
Moreover, whenever $\vJ$ is uniquely determined, there are explicit reconstruction procedures.
\end{proposition}

\begin{proof}
If $\Omega\subseteq \supp(\omega^2 \varepsilon \gamma_\varepsilon + i\omega \sigma \gamma_\sigma)$,
then \eqref{eq:intvector} implies
$$
\vE = \frac{\vQ}{\omega^2 \varepsilon \gamma_\varepsilon + i\omega \sigma \gamma_\sigma}.
$$
This uniquely determines the weak solution $\vE\in X$ everywhere in $\Omega$. Consequently, $\vH$ and $\vJ$ are also uniquely determined via the Maxwell's equations~\eqref{eq:frequency}. Note that all these procedures are constructive: given $\vQ$, we compute $\vE$ from the above equality, and then $\vJ$ from~\eqref{eq:frequency}.

The stability can be derived as follows. If there is another source $\tilde{\vJ}$ with corresponding electric field $\tilde{\vE}$, and vector internal functional $\tilde{\vQ}$ defined as in~\eqref{eq:intvector}, then $\vE - \tilde{\vE}$ is a weak solution of the Maxwell equations. That is,
\begin{align*}
\left(\frac{1}{\mu} \cl (\vE - \tilde{\vE}), \cl\vphi \right)_{\Ltc}  - \left((\omega^2\varepsilon+ i\omega\sigma)(\vE - \tilde{\vE}), \vphi\right)_{\Ltc} \\
- i\omega\langle \lambda (\vE - \tilde{\vE})_T, \vphi_T\rangle= i\omega \left(\vJ - \tilde{\vJ}, \vphi\right)_{\Ltc}
\end{align*}
for all $\vphi\in X$. As the coefficients in this weak formulation are all bounded, there exists a constant $C>0$ such that 
$$
\left|\omega \left(\vJ - \tilde{\vJ}, \vphi\right)\right| \leq C \|\vE - \tilde{\vE}\|_X  \; \|\phi\|_X.
$$
We deduce that
$$\|\vJ - \tilde{\vJ}\|_{X^*} \leq C \|\vE - \tilde{\vE}\|_X = 
C \left\| \frac{\vQ - \tilde{\vQ}}{\omega^2 \varepsilon \gamma_\varepsilon + i\omega \sigma \gamma_\sigma} \right\|_X.
$$

If $\Omega\nsubseteq \supp(\omega^2 \varepsilon \gamma_\varepsilon + i\omega \sigma \gamma_\sigma)$, there exists an open set $D\subseteq \Omega\backslash \supp(\omega^2 \varepsilon \gamma_\varepsilon + i\omega \sigma \gamma_\sigma)$. For any compactly supported smooth function
$\phi\in C^\infty_c(D)$, if $(\vE,\vH)$ solves~\eqref{eq:frequency}, then $(\vE+\nabla\phi,\vH)$ solves~\eqref{eq:frequency} with $\vJ$ replaced by $\vJ+(i\omega\varepsilon-\sigma)\nabla\phi$. Moreover, since $(\vE+\nabla\phi,\vH)|_{\partial\Omega} = (\vE,\vH)$, these two pairs both satisfy the boundary condition~\eqref{eq:imp} and produce identical measurement~\eqref{eq:meas}. This means that sources of the form
$\vJ_\phi := (i\omega\varepsilon-\sigma)\nabla\phi$ are non-radiating. Thus the source $\vJ$ cannot be uniquely determined from the boundary measurement~\eqref{eq:meas}.
\end{proof}

\subsubsection{Increased Regularity}
The stability estimate for the subcase $\gamma_J=0$ and $\Omega\subseteq \supp(\omega^2 \varepsilon \gamma_\varepsilon + i\omega \sigma \gamma_\sigma)$
is in terms of the $X^*$ norm, which follows because $\vJ$ was obtained from $\vE$ using a weak formulation. When the reconstructed $\vE$ and $\vH$ are smooth enough, for example, when $\vE\in(H^2(\Omega))^3$, we can utilize the strong formulation to control $\vJ$ in $\Ltc$ in terms of the higher order derivatives of the internal data.

\begin{proposition}
\label{lem:regC5}
Suppose the assumptions A1-A4 and the hypothesis~\eqref{hypo:indep} hold. Suppose, in addition, that $\varepsilon,\sigma\in C^{1,1}(\Omega)$.
If $\gamma_J=0$ and $\Omega\subseteq \supp(\omega^2 \varepsilon \gamma_\varepsilon + i\omega \sigma \gamma_\sigma)$, then the following stability estimate holds for any two compactly supported sources $\vJ, \tilde{\vJ}\in (H^2(\Omega))^3$:
$$
\|\vJ-\tilde{\vJ}\|_{\Ltc} \leq C \left\|\frac{\vQ - \tilde{\vQ}}{\omega^2 \varepsilon \gamma_\varepsilon + i\omega \sigma \gamma_\sigma}\right\|_{(H^2(\Omega_1))^3}.
$$
Here $\Omega_1$ is an open set compactly contained in $\Omega$ such that $\supp{\vJ}\subset\Omega_1$ and $\supp{\tilde{\vJ}}\subset\Omega_1$, and the constant $C>0$ is independent of $\vJ, \tilde{\vJ}$.
\end{proposition}
\begin{proof}
Define
$$\vu:=\vE - \tilde{\vE} = \frac{\vQ - \tilde{\vQ}}{\omega^2\varepsilon\gamma_\varepsilon + i\omega\sigma\gamma_\sigma} .
$$
Then $\vu$ solves
\begin{equation} \label{eq:uu}
\nabla\times\frac{1}{\mu}\nabla\times \vu - (\omega^2\varepsilon + i\omega\sigma) \vu = i\omega (\vJ-\tilde{\vJ}).
\end{equation}
The following stability estimate is immediate:
$$
\|\vJ-\tilde{\vJ}\|_{\Ltc)} = \|\vJ-\tilde{\vJ}\|_{(L^2(\Omega_1))^3} \leq C \|\vu\|_{(H^2(\Omega_1))^3} =  C\left|\frac{\vQ - \tilde{\vQ}}{\omega^2 \varepsilon \gamma_\varepsilon + i\omega \sigma \gamma_\sigma}\right\|_{(H^2(\Omega_1))^3}.
$$
It remains to show that the quantity 
$$
\|\vu\|_{(H^2(\Omega_1))^3} = \left\|\frac{\vQ - \tilde{\vQ}}{\omega^2 \varepsilon \gamma_\varepsilon + i\omega \sigma \gamma_\sigma}\right\|_{H^2(\Omega_1)}
$$ 
is finite. To proceed we will employ an interior regularity estimate for elliptic equations.
Denote by $a:=\omega^2\varepsilon >0$, $b:=i\omega\sigma \geq 0$, $\vf:= i\omega (\vJ-\tilde{\vJ})$. Then take the divergence of~\eqref{eq:uu} to obtain
\begin{equation} \label{eq:divergence}
 \nabla\cdot \vu = - \frac{\nabla (a+i b) \cdot \vu +  \nabla\cdot \vf }{a+i b}. 
\end{equation}
Using the identity $\cl\cl\vu=\nabla(\nabla\cdot\vu)-\Delta\vu$, we obtain the following elliptic system
$$
\Delta\vu = - \nabla \left[ \frac{\nabla (a+i b) \cdot \vu +  \nabla\cdot \vf }{a+i b} \right] - \mu (a+i b) \vu - \mu \vf. 
$$
For each component of $\vu$, the left hand side of the above defines a second order elliptic operator with constant coefficients, so we can apply an interior regularity estimate~\cite[Lemma 6.32]{Folland1995}. Since $a$ is bounded away from zero and $a,b$ are piecewise $C^{1,1}(\Omega)$, the following quantities are all bounded in~$\overline{\Omega}$:
$$
|a+i b|, \  \left| \nabla\cdot\frac{\nabla (a+i b)}{a+i b}\right|, \ \left| \frac{\nabla (a+i b)}{a+i b}\right|, \  \left| \nabla \frac{1}{a+i b}\right|, \  \left| \frac{1}{a+i b}\right|.
$$
Thus
$$
\begin{aligned}
\nabla \left( \frac{\nabla (a+i b) \cdot \vu +  \nabla\cdot \vf }{a+i b} \right) + \mu (a+i b) \vu + \mu \vf  \in (H_{\text{loc}}^{-1}(\Omega))^3.
\end{aligned}
$$
Combing this with the fact that $\vu\in\Ltc$, we obtain from~\cite[Lemma 6.32]{Folland1995} that $\vu\in (H_{\text{loc}}^{1}(\Omega))^3$. 
This increased regularity implies that   
$$
\begin{aligned}
\nabla \left[ \frac{\nabla (a+i b) \cdot \vu +  \nabla\cdot \vf }{a+i b} \right] +& \mu (a+i b) \vu + \mu \vf  \in {(L_{\text{loc}}^2(\Omega))^3}.
\end{aligned}
$$
Applying~\cite[Lemma 6.32]{Folland1995} again, we obtain that $\vu\in (H_{\text{loc}}^{2}(\Omega))^3$. 
Next, choose a smooth cutoff function $\chi$ that is compactly supported in $\Omega$ and equal to one on $\Omega_1$. We find that 
$$
\|\vu\|_{(H^2(\Omega_1))^3} \leq \|\chi\vu\|_{(H^2(\R^3))^3} <\infty.
$$
Thus we obtain that $\vu \in H^2(\Omega_1)$. 
\end{proof}

\subsection{Case II: $\gamma_J\neq 0$.}
Here~\eqref{eq:intvector} implies that
\begin{equation} \label{eq:JQE}
\vJ = \frac{\vQ - \left(\omega^2 \varepsilon \gamma_\varepsilon + i\omega \sigma \gamma_\sigma \right) \vE}{i\omega\gamma_J}.
\end{equation}
Inserting the above into the Maxwell equations~\eqref{eq:frequency}, we obtain an equation of the form
\begin{equation} \label{eq:Eeq}
\nabla\times\frac{1}{\mu}\nabla\times \vE - (a+i b) \vE = \frac{\vQ}{\gamma_J} ,
\end{equation}
where 
$$
a= \omega^2\varepsilon \left(1-\frac{\gamma_\varepsilon}{\gamma_J}\right), \quad\quad
b= \omega\sigma \left(1-\frac{\gamma_\sigma}{\gamma_J}\right).
$$
Note that there are boundary constraints for $\vE$, including the impedance boundary condition~\eqref{eq:imp} and the measurement~\eqref{eq:meas}.

To analyze the stability of the inverse problem, suppose that there is another source $\tilde{\vJ}$ with corresponding electric field $\tilde{\vE}$ and vector internal functional $\tilde{\vQ}$, defined by~\eqref{eq:intvector}. Let $\vu:= \vE - \tilde{\vE} \in X$ be a weak solution of the equation
\begin{equation} \label{eq:ueq}
\nabla\times\frac{1}{\mu}\nabla\times \vu - (a+i b) \vu = \frac{\vQ - \tilde{\vQ}}{\gamma_J} \quad\quad \text{ in } \quad\Omega;
\end{equation}
and obey the impedance boundary condition
\begin{equation} \label{eq:uimp}
\left(\frac{1}{\mu} \cl \vu\right)\times \vn - i \omega\lambda (\vn\times\vu)\times\vn=0  \quad \text{on } \quad\partial\Omega ,
\end{equation}
where
\begin{equation} \label{eq:umeas}
(\vn\times \vu)\times\vn|_{\partial\Omega} = \vg - \tilde{\vg}.
\end{equation}

It remains to establish the solvability of \eqref{eq:Eeq} with boundary condition \eqref{eq:imp}  or the solvability of \eqref{eq:ueq} with boundary condition \eqref{eq:uimp}.
Now \eqref{eq:Eeq} and \eqref{eq:ueq} are similar in form to~\eqref{eq:forward}. The difference is that in~\eqref{eq:forward}, the term $(\omega^2\varepsilon_\delta+i\omega\sigma_\delta)$ has a strictly positive real part and a non-negative imaginary part, which ensures the existence and uniqueness of the weak solution by standard methods~\cite{Monk2003}. These sign conditions no longer hold for the term $a+ib$ in~\eqref{eq:Eeq} and \eqref{eq:ueq}, due to the presence of the elasto-electric constants $\gamma_\varepsilon, \gamma_\sigma,\gamma_J$.
Therefore, we divide the discussion into several sub-cases.  
By Assumption A-3, we see that $a$ is either identically zero or bounded away from zero, $b$ is either non-positive or non-negative. This observation accounts for the following classification of sub-cases.

\begin{theorem} 
\label{thm:case2}
Suppose the assumptions A1-A4 and the hypothesis~\eqref{hypo:indep} hold. If $\gamma_J\neq 0$, we have the following subcases:
\begin{enumerate}
    \item[(II.1)] If $\gamma_\varepsilon = \gamma_\sigma = \gamma_J$, then the source $\vJ$ cannot be uniquely determined. 
    \item[(II.2)] If $\gamma_\varepsilon = \gamma_J$, $\gamma_\sigma \neq \gamma_J$, and $\Omega\subseteq\supp{\sigma}$, then the source $\vJ$ is uniquely determined. If {in addition $|b|$ is strictly bounded away from zero}, then we have the following stability estimates. If ${\gamma_\sigma}/{\gamma_J}<1$,
$$
\|\vJ-\tilde{\vJ}\|_{\Ltc} \leq C \|\vQ-\tilde{\vQ}\|_{\Ltc} ,
$$
and if ${\gamma_\sigma}/{\gamma_J}>1$,
$$
\|\vJ-\tilde{\vJ}\|_{\Ltc} \leq C (\|\vQ-\tilde{\vQ}\|_{\Ltc} + \|\vg-\tilde{\vg}\|_{\Ltcb}).
$$
    \item[(II.3)] If $\gamma_\varepsilon = \gamma_J$, $\gamma_\sigma \neq \gamma_J$, and $\Omega\not\subseteq\supp{\sigma}$, then the source $\vJ$ cannot be uniquely determined.
    \item[(II.4)] If $\gamma_\varepsilon\neq\gamma_J$, then the source $\vJ$ is uniquely determined. If ${\gamma_\varepsilon}/{\gamma_J}>1$, we have the following stability estimate
$$
\|\vJ-\tilde{\vJ}\|_{\Ltc} \leq C \|\vQ-\tilde{\vQ}\|_{\Ltc},
$$
and if ${\gamma_\varepsilon}/{\gamma_J}<1$, 
$$
\|\vJ-\tilde{\vJ}\|_{\Ltc} \leq C (\|\vQ-\tilde{\vQ}\|_{\Ltc} + \|\vg-\tilde\vg\|_{\Ltcb}).
$$
\end{enumerate}
Here $C>0$ is a constant independent of $\vJ, \tilde{\vJ}$. Moreover, whenever $\vJ$ is uniquely determined, there are explicit reconstruction procedures.
\end{theorem}
The proof is presented in the next few subsections.

\begin{remark}
    It is generally expected that $\gj\neq \geps$ because the former is solely a density effect, and the latter is due to density variation and Brillouin scattering~\cite{Bal_2010}. Thus Case $(II.4)$ is more likely to occur in practice.
\end{remark}

The results in Proposition~\ref{thm:case1} and Theorem~\ref{thm:case2} can be summarized by the following table:

\begin{table}[h]
\begin{tabular}{|c|l|c|}
\hline
Case & \hspace{2.5cm} Subcase &  Uniqueness \\
\hline
$\gamma_J=0$ & (I.1) $\Omega\subseteq \supp(\omega^2 \varepsilon \gamma_\varepsilon + i\omega \sigma \gamma_\sigma)$   &  Y  \\
 & (I.2) $\Omega\not\subseteq \supp(\omega^2 \varepsilon \gamma_\varepsilon + i\omega \sigma \gamma_\sigma)$  &  N  \\
 \hline
$\gamma_J \neq 0$ & (II.1) $\gamma_\varepsilon = \gamma_\sigma = \gamma_J$  &  N  \\
 & (II.2) $\gamma_\varepsilon = \gamma_J$, $\gamma_\sigma \neq \gamma_J$, $\Omega\subseteq\supp{\sigma}$, $\sigma >0$ &  Y  \\
  & (II.3) $\gamma_\varepsilon = \gamma_J$, $\gamma_\sigma \neq \gamma_J$, $\Omega\not\subseteq\supp{\sigma}$  &  N  \\
 & (II.4) $\gamma_\varepsilon \neq \gamma_J$ &  Y  \\
 \hline
\end{tabular}
\end{table}

\bigskip
\subsubsection{Subcase (II.1): $\gamma_\varepsilon = \gamma_\sigma = \gamma_J$.}
This subcase corresponds to $a\equiv 0$ and $b\equiv 0$. For any $\phi\in C^\infty_c(\Omega)$, if $\vE$ satisfies the equation~\eqref{eq:Eeq} and the boundary condition~\eqref{eq:imp}, so does $\vE + \nabla\phi$. This means, as a result of~\eqref{eq:intvector}, that sources of the form $J_\phi := (i\omega\varepsilon-\sigma)\nabla\phi$ are non-radiating. Thus the original source $\vJ$ cannot be uniquely determined.

\bigskip
\subsubsection{Subcase (II.2 and II.3): $\gamma_\varepsilon = \gamma_J$ and $\gamma_\sigma \neq \gamma_J$.}
This subcase corresponds to $a\equiv 0$ and $b \not\equiv 0$. Note that either $b\geq 0$ everywhere or $b \leq 0$ everywhere due to the assumption A-3. 
In the following discussion, we will keep $a$ as a placeholder, for ease of exposition. 

We now take the inner product of~\eqref{eq:ueq} with $\vu^*$ and use the vector identity $(\nabla\times \vA)\cdot \vB = \nabla\cdot (\vA\times \vB) + (\nabla\times \vB)\cdot \vA$ to integrate by parts:
\begin{align}
\nonumber 
\label{eq:ibp}
&\int_{\partial\Omega} \left[
\frac{1}{\mu}( \nabla\times \vu) \times \vu^* \right]\cdot \vn \, d {\bf x} + \int_\Omega \left[
\frac{1}{\mu} |\nabla\times \vu|^2 - (a+ib) |\vu|^2 \right] \, d {\bf x}  \\
&= \frac{1}{\gamma_J}\int_\Omega (\vQ - \tilde{\vQ})\cdot \vu^* \, d {\bf x}.
\end{align}
For the boundary integral, we apply the vector triple product identity $(\vB\times\vC)\cdot\vA 
= \vC\cdot(\vA\times\vB)$ to obtain
$$
\left[
\frac{1}{\mu} (\nabla\times \vu) \times \vu^* \right]\cdot \vn 
= \vu_T^* \cdot  \left[ \vn \times \left( \frac{1}{\mu} \nabla\times \vu \right)\right]
\quad  \text{ on } \partial\Omega.
$$
Here $ \vu_T= (\vn\times \vu)\times\vn|_{\partial\Omega},$ is the tangential trace of $\vu$ as defined in \eqref{eq:tangtrace}. 
From~\eqref{eq:uimp},
we obtain $\vn\times(\frac{1}{\mu} \nabla\times \vu) =  -i\omega\lambda \vu_T$.
Thus the boundary integrand becomes
$$
\left[
(\frac{1}{\mu} \nabla\times \vu) \times \vu^* \right]\cdot \vn = 
\vu_T^* \cdot  \left[ \vn \times \left( \frac{1}{\mu} \nabla\times \vu \right)\right] =   -i\omega\lambda |\vu_T|^2
= -i\omega\lambda |\vg - \tilde{\vg}|^2.
$$
Therefore, separating the real and imaginary parts of~\eqref{eq:ibp} we obtain
\begin{align}
\int_\Omega
\frac{1}{\mu} |\nabla\times \vu|^2 - a |\vu|^2 \, d {\bf x}
& = \int_\Omega \Re \left[ \frac{(\vQ - \tilde{\vQ})\cdot \vu^*}{\gamma_J} \right] \, d {\bf x} , \label{eq:ibpreal}\\
\int_\Omega b |\vu|^2 \, d {\bf x} + \int_{\partial\Omega} \omega\lambda |\vg - \tilde{\vg}|^2 \, d {\bf x} 
& = 
- \int_\Omega \Im \left[ \frac{(\vQ - \tilde{\vQ})\cdot \vu^*}{\gamma_J} \right] \, d {\bf x}.\label{eq:ibpimag}
\end{align}

To prove uniqueness, we set $\vg = \tilde{\vg}$. Then $\vQ = \tilde{\vQ}$ and~\eqref{eq:ibpimag} implies $b|\vu|=0$.
If $\Omega\subseteq \supp\sigma = \supp b$. 
We conclude that $\vu\equiv 0$ in $\Omega$. If $\Omega\not\subseteq \supp\sigma = \supp b$,
there exists an open set $D\subseteq \Omega\backslash \supp{b}$. For any compactly supported smooth function
$\phi\in C^\infty_c(D)$, the choice $\vu:=\nabla\phi$ is a non-trivial solution to~\eqref{eq:ueq} \eqref{eq:uimp}, proving the non-uniqueness.

Now we prove stability assuming that $\sigma$ is strictly positive, which implies that $b$ is bounded away from zero.
When $b<0$, recall that $a=0$, so there exists a constant $C>0$, independent of $\vu$, such that
\begin{align}
\|\vu\|^2_{H(curl,\Omega)}  \leq &
C \left( \int_\Omega
\frac{1}{\mu} |\nabla\times \vu|^2 - a |\vu|^2 \, d {\bf x}
+ \int_\Omega b |\vu|^2 \, d {\bf x} \right) \nonumber \\
\leq & C (\|\vQ-\tilde{\vQ}\|_{\Ltc} \|\vu\|_{\Ltc} + \|\vg-\tilde{\vg}\|^2_{\Ltcb}) \label{eq:ineq1}\\
\leq & C (\frac{\eta}{2}\|\vQ-\tilde{\vQ}\|^2_{\Ltc} + \frac{1}{2\eta} \|\vu\|^2_{\Ltc} + \|\vg-\tilde{\vg}\|^2_{\Ltcb}) \nonumber 
\end{align}
where $\eta>0$ is an arbitrary constant. If we choose $\eta$ so that $\frac{C}{2\eta}<1$, then the term $\frac{C}{2\eta} \|\vu\|_{\Ltc}$ can be absorbed into the left hand side, resulting in the following estimate (with a different constant $C$):
$$
\|\vu\|_{H(curl,\Omega)} \leq C (\|\vQ-\tilde{\vQ}\|_{\Ltc} + \|\vg-\tilde{\vg}\|_{\Ltcb}).
$$
This result, combined with~\eqref{eq:JQE}, yields the stability estimate 
$$
\|\vJ-\tilde{\vJ}\|_{\Ltc} \leq C (\|\vQ-\tilde{\vQ}\|_{\Ltc} + \|\vg-\tilde{\vg}\|_{\Ltcb}). 
$$

When $a=0$ and $b>0$ is bounded away from zero, we can obtain a better stability estimate. In this case, the left hand side of~\eqref{eq:ibpimag} is the sum of two non-negative terms, then the estimate~\eqref{eq:ineq1} can be improved as
\begin{align*}
\|\vu\|^2_{H(curl,\Omega)}  \leq &
C \left( \int_\Omega
\frac{1}{\mu} |\nabla\times \vu|^2 - a |\vu|^2 \, d {\bf x}
+ \int_\Omega b |\vu|^2 \, d {\bf x} \right) \\
\leq & C \|\vQ-\tilde{\vQ}\|_{\Ltc} \|\vu\|_{\Ltc} \\
\leq & C \|\vQ-\tilde{\vQ}\|_{\Ltc} \|\vu\|_{H(curl,\Omega)}.
\end{align*}
Canceling out $\|\vu\|^2_{H(curl,\Omega)}$ and applying the relation~\eqref{eq:JQE} yields the stability estimate 
$$
\|\vJ-\tilde{\vJ}\|_{\Ltc} \leq C \|\vQ-\tilde{\vQ}\|_{\Ltc}. 
$$

\subsubsection{Subcase (II.4): $\gamma_\varepsilon \neq  \gamma_J$.}
This subcase corresponds to $a\neq 0$. Note that due to the assumption A-3, there exists a constant $c>0$ such that either $a \geq c > 0$ everywhere or $a \leq -c < 0$ everywhere in $\Omega$.

$\bullet$ If $a\leq -c < 0$, the identities~\eqref{eq:ibp} \eqref{eq:ibpreal} \eqref{eq:ibpimag} still hold, hence there exists a constant $C>0$ such that
$$
\|\vu\|^2_{L^2(\Omega)}  \leq C  \int_\Omega
\frac{1}{\mu} |\nabla\times \vu|^2 - a |\vu|^2 \, d {\bf x}
= C \int_\Omega \Re \left[ \frac{(\vQ - \tilde{\vQ})\cdot \vu^*}{\gamma_J} \right] \, d {\bf x}.
$$
where the second equality comes from~\eqref{eq:ibpreal}.
Suppose  $\vQ = \tilde{\vQ}$, then $\vu=0$ in $\Omega$, proving uniqueness. The above inequality also implies, by the Cauchy-Schwartz inequality, that
$$
\|\vu\|^2_{\Ltc}  \leq C \|\vQ-\tilde{\vQ}\|_{\Ltc} \|\vu\|_{\Ltc}.
$$
Canceling  factors of $\|\vu\|_{L^2(\Omega)}$ and applying the relation~\eqref{eq:JQE} yields the stability estimate 
$$
\|\vJ-\tilde{\vJ}\|_{\Ltc} \leq C \|\vQ-\tilde{\vQ}\|_{\Ltc}. 
$$

$\bullet$ If $a\geq c > 0$, we consider $\vu$ equipped with the Dirichlet boundary condition 
\begin{equation}\label{eq:traceD}
\vdg := \vu_T = (\vn\times \vu)\times\vn  =\tilde{\vg}-\vg.   
\end{equation}
Since $\vE,\tilde{\vE}\in X$, we conclude that $\vdg\in\Hmhcl$. Thus, there exists a function $\mathcal G\in\Hc$, such that 
$$
\mathcal G_T =\vdg\quad \text{and} \quad \| \mathcal G\|_{\Hc} \leq C\|\vdg\|_{\Hmhcl}.
$$
Set $\tilde{\vu}:=\vu-\mathcal{G}$, then $\tilde{\vu}_T=0$ and $\tilde{\vu}$ solves
\begin{equation}
\label{eq:weakDH}
\left(\frac{1}{\mu} \cl \tilde\vu, \cl\vphi\right)_{\Ltc)}  - \left( (a+ib)\tilde\vu, \vphi\right)_{\Ltc} 
 =   \left(\tilde\vf, \vphi\right)_{\Ltc}, \ \vphi\in \Hc_0,
\end{equation}
where $\Hc_0$ is the subspace of $\Hc$ with zero tangential trace, and
$$
\tilde\vf := \frac{i\omega}{\gamma_J} (\vQ-\tilde{\vQ}) + \nabla\times (\frac{1}{\mu} \cl \mathcal G) - (a+ib)\mathcal G.
$$
Note that for $\tilde\vf\in (\Hc_0)^*$, the dual space of $\Hc_0$, we have
$$
\|\tilde\vf\|_{(\Hc_0)^*}\leq C(\|\vQ-\tilde{\vQ}\|_{\Ltc}+\|\mathcal G \|_{\Hc} \leq C(\|\vQ-\tilde{\vQ}\|_{\Ltc)}+ \|\vdg\|_{\Hmhcl}).
$$
It follows from~\cite[Theorem 4.17]{Monk2003} that there exists a unique solution $\tilde\vu$ to~\eqref{eq:weakDH} with 
\begin{align*}
\|\tilde\vu\|_{\Hc} & \leq  C(\|\tilde{\vf}\|_{(\Hc_0)^*} + \|\vdg\|_{\Hmhcl}) \\
 & \leq C(\|\vQ-\tilde{\vQ}\|_{\Ltc}+ \|\vg-\tilde{\vg}\|_{\Ltcb}).
\end{align*}

\begin{remark}
Note that when applying~\cite[Theorem 4.17]{Monk2003}, the entire boundary $\partial \Omega$ is equipped with the homogeneous Dirichlet condition. {Here we have extended~\cite[Theorem 4.17]{Monk2003}, which requires $b\geq0$, but is obviously correct for $b\leq0$ when all of $\partial\Omega$ has homogeneous Dirichlet boundary condition and $b$ is not constantly zero.}

\end{remark}

Finally, whenever $\vJ$ is uniquely determined, it can be reconstructed as follows. Given $\vQ$ and $\vg$, solve the boundary value problem~\eqref{eq:Eeq} and \eqref{eq:meas} to obtain $\vE$. Then use the Maxwell's equation~\eqref{eq:frequency} to recover $\vJ$.

\section{Numerical Experiments}

In this section, we present numerical experiments to test the reconstruction of $\vJ$ in Case (I.1) and Case (II.4). The code is implemented in Python using the finite element PDE solver \texttt{NGSolve} \footnote{The code is hosted at \href{https://github.com/lowrank/umme}{https://github.com/lowrank/umme} }. Numerical experiments are performed on the domain consisting of an infinite cylinder of radius $r= 1\text{cm}$, discretized with a uniform triangular mesh of 19276 triangles. The Maxwell equations \eqref{eq:frequency} are solved with a third-order N\'ed\'elec element.

We denote by $\varepsilon_0$ and $\mu_0$ the electric permittivity and the magnetic permeability in vacuum, respectively. In a medium with 
electric permittivity $\varepsilon$ and magnetic permeability $\mu$, we define
$$
\varepsilon_r := \frac{\varepsilon}{\varepsilon_0},\quad\quad 
\mu_r = \frac{\mu}{\mu_0}.
$$
We refer to $\varepsilon_r$ and $\mu_r$ as the relative electric permittivity and the relative magnetic permeability, respectively. Moreover,
let $c$ be the light speed in  vacuum and define
\begin{equation}
    \hat{\sigma} = \frac{1}{c\eps_0}\sigma ,\quad \hat{\bJ} = c\mu_0 \bJ,\quad \hat\omega = \frac{\omega}{c}.
\end{equation}
Using the relation $c={1}/{\sqrt{\varepsilon_0\mu_0}}$, we can rewrite the Maxwell equations~\eqref{eq:frequency} as 
$$
    \nabla\times  \frac{1}{\mu_r} \nabla\times \vE - (\hat{\omega}^2 \eps_r + i\hat\omega \hat\sigma)\vE = i \hat\omega \hat{\bJ} ,
$$
together with the impedance boundary condition~\eqref{eq:imp}, with impedance $\lambda = 1$. 

The physical parameters are chosen as follows. According to Assumption A-2, $\mu_r =1$. The frequency is selected such that $\hat \omega = \pi [\text{cm}^{-1}]$, which corresponds to a frequency $f = \frac{\omega}{2\pi} \approx  15$GHz.
Density plots of $\varepsilon_r$ and $\hat{\sigma}$ are displayed in Fig.~\ref{fig:epsilonsigma}.
For $\varepsilon_r$, the background value is taken to be 37.2 for blood  (see \cite{gabriel1996compilation}) and there are 3 regions with smaller values of $\eps_r$ which are 7.79 (top) for fat, 20.2 (left) for nerve and 36.4 (right) for muscle. The source $\hat{\vJ}$ is a real vector, whose components are shown in Fig.~\ref{fig:J}.

\bigskip

\begin{figure}[!htb] 
    \centering
    \includegraphics[width=0.48\textwidth, height=0.38\textwidth]{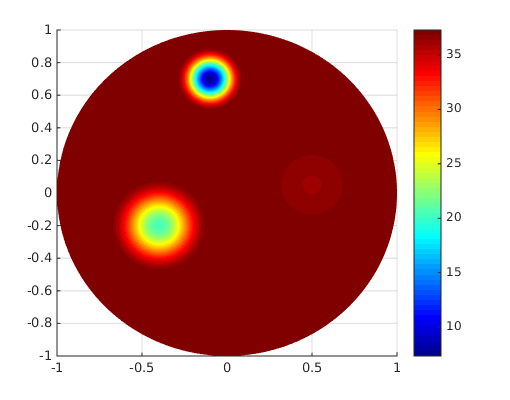}
    \includegraphics[width=0.46\textwidth, height=0.38\textwidth]{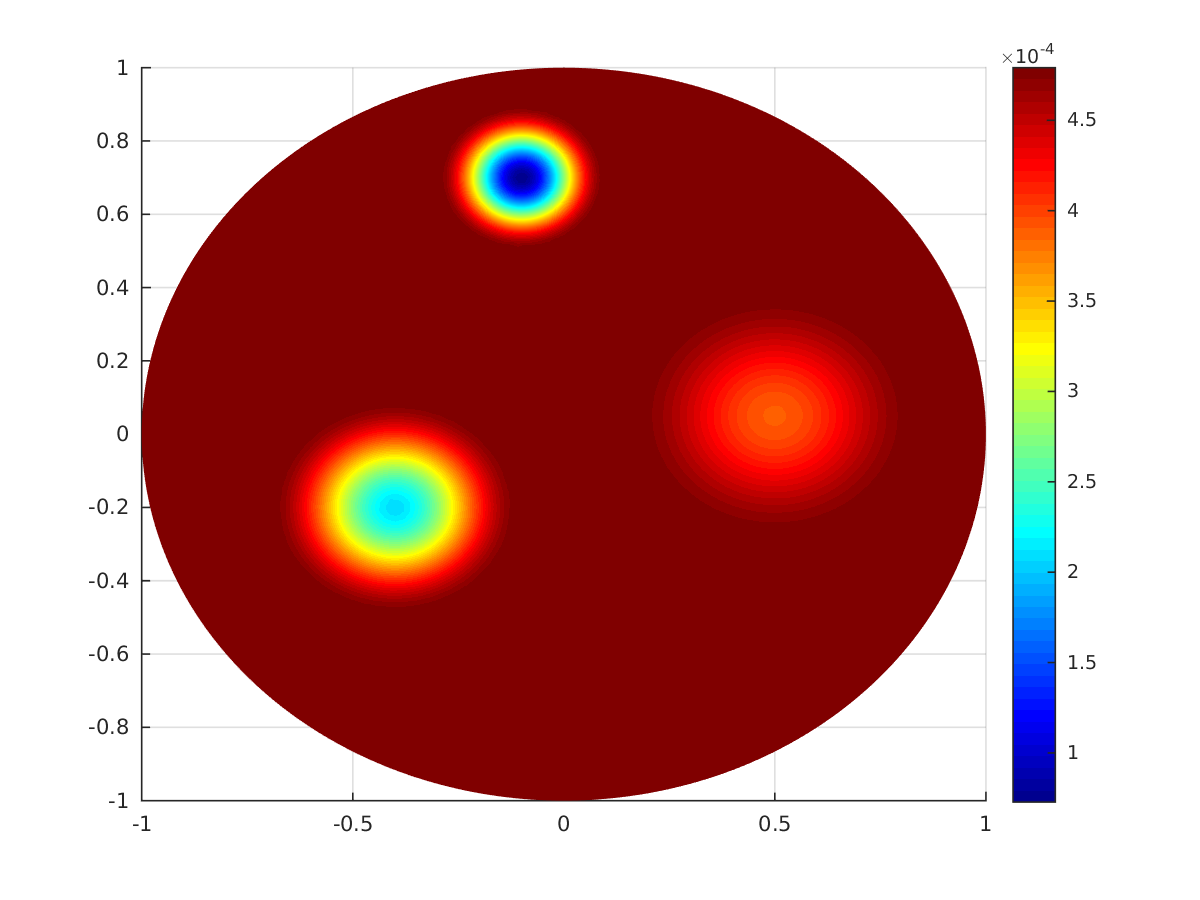}
    \caption{Left: $\varepsilon_r$, Right: $\hat \sigma [\text{cm}^{-1}]$} \label{fig:epsilonsigma}
\end{figure}

\begin{figure}[!htb]
    \centering
    \includegraphics[width=0.48\textwidth, height=0.38\textwidth]{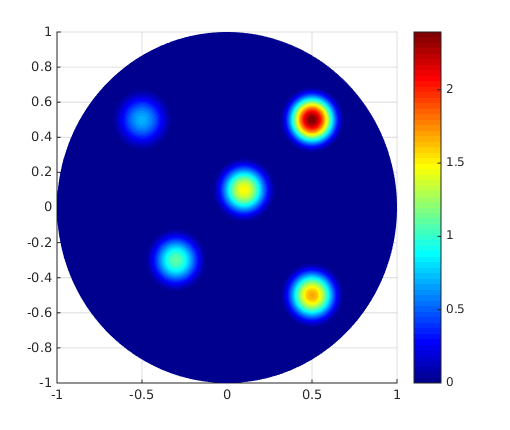}
    \includegraphics[width=0.48\textwidth, height=0.38\textwidth]{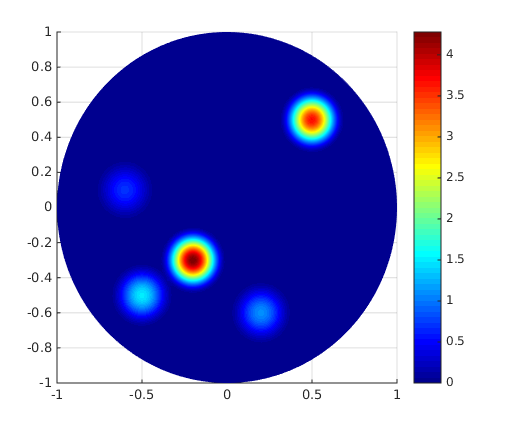}
    \caption{Source $\hat \bJ$. Left: $x$-component, Right: $y$-component. }
    \label{fig:J}
\end{figure}

Auxiliary solutions are needed in the reconstruction to compute the vector internal data~\eqref{eq:intvector} from the scalar internal data~\eqref{eq:intscalar}.
Such solutions are obtained by solving the equation for $j=1,2$:
$$
\cl \cl \vF_j^* - (\hat{\omega}^2 \varepsilon_r + i\hat{\omega}\hat{\sigma}) \vF_j^*  = 0 \quad \text{in } \quad \Omega ,
$$
along with the impedance boundary condition
$$
\left(\frac{1}{i\hat{\omega}} \nabla\times \vF_j^*\right) \times \vn   - \lambda (\vn\times\vF_j^*)\times\vn = \mathfrak{g}_j  \quad\quad \text{on } \quad\partial\Omega.
$$
where $\mathfrak{g}_k$ is defined by
\begin{equation}
    \mathfrak{g}_j := \frac{1}{i\hat{\omega}}\nabla\times \bE_j \times \bn - \lambda (\bn\times \bE_{j}) \times \bn \; \text{ on }\partial\Omega , 
\end{equation}
with $\bE_1 = (e^{-iky}, 0)$ and $\bE_2 = (0, e^{-ikx})$. Here the wave number $k=\sqrt{(\hat{\omega}^2\eps_r + i\hat{\omega}\hat{\sigma})}$, where $\eps_r, \hat{\sigma}$ are taken from the background values corresponding to blood. 
The rationale for the choice of $\mathfrak{g}_j$ is that when the medium is homogeneous, then $\bE_1$ and $\bE_2$ are mutually orthogonal plane waves. Clearly, such an orthogonality relation may not hold in practice due to the distortion caused by the inhomogeneity.
\subsection{Case (I.1)}
In this experiment, the modulation parameters are chosen as 
$\gamma_J=0$, $\gamma_{\eps} = 0.25$, $\gamma_{\sigma}=0.35$.
The scalar internal data $Q$ is obtained by solving the forward problem~\eqref{eq:frequency}.
Then $0.1\%$ multiplicative noise is added to the signal.
The vector internal data $\vQ$ is calculated from the auxiliary solutions.
The reconstruction is carried out using the procedure described in Proposition~\ref{thm:case1}. That is, we solve for $\vE$ from~\eqref{eq:intvector} and then recover $\vJ$ from the Maxwell equations~\eqref{eq:frequency} through the following weak formulation by setting $\delta=0$ in~\eqref{eq:weak}:
\begin{align}\nonumber
  &i\omega \left(\vJ, \vphi\right)_{\Ltc}\\&=\left(\frac{1}{\mu} \cl \vE, \cl\vphi \right)_{\Ltc}  - \left((\omega^2\varepsilon+ i\omega\sigma)\vE, \vphi\right)_{\Ltc} - i\omega\langle \lambda \vE_{T}, \vphi_T\rangle.  \nonumber
  \end{align}
  Here $\vJ$ is solved under the Galerkin framework by treating the left-hand side $i\omega \left(\vJ, \vphi\right)_{\Ltc}$ as the bilinear form and the right-hand side as the linear form with known $\vE$.
The reconstructed source is shown in Fig.~\ref{fig:caseI1}.

\begin{figure}[!htb]
    \centering
    \includegraphics[width=0.46\textwidth, height=0.38\textwidth]{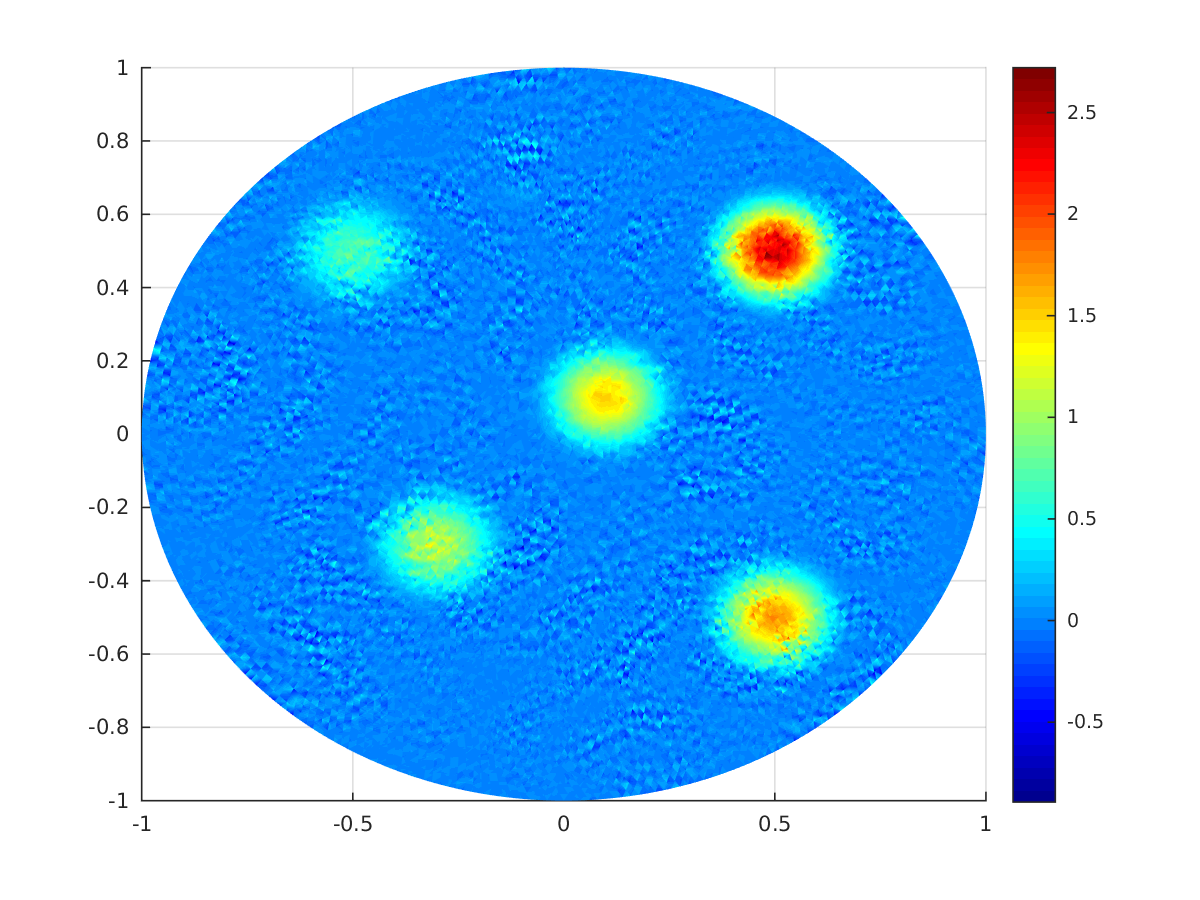}
    \includegraphics[width=0.46\textwidth, height=0.38\textwidth]{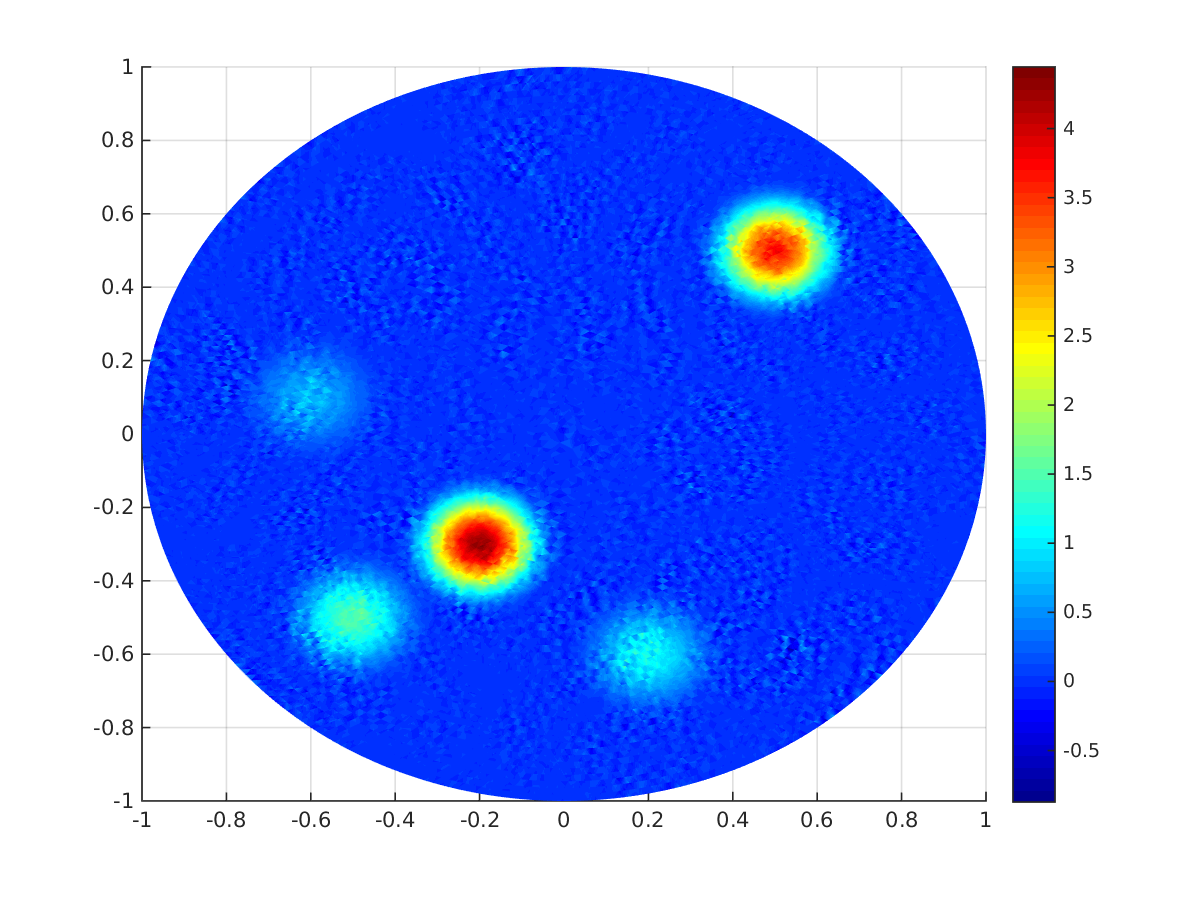}\\
    \includegraphics[width=0.46\textwidth, height=0.38\textwidth]{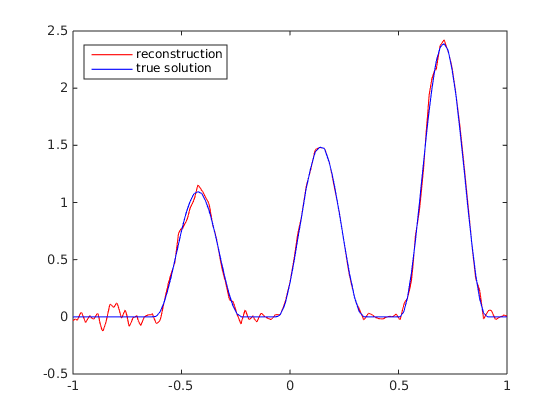}
    \includegraphics[width=0.46\textwidth, height=0.38\textwidth]{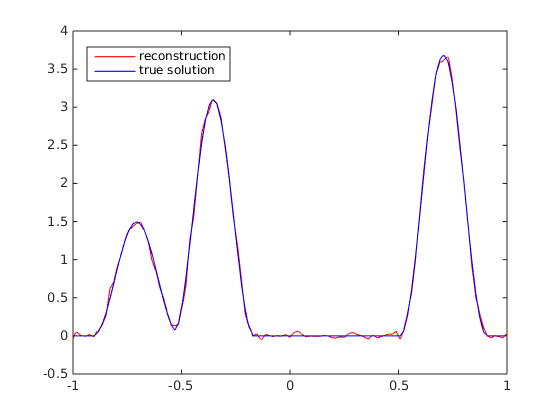}
    \caption{Reconstructed source current vector for Case (I.1). Top row: from left to right, real part of the $x$ component and the $y$ component. Bottom row: from left to right, cross sections of the reconstruction and the true solution along the line $y=x$. The relative $L^2$ error of reconstruction is $32.5\%$.}
    \label{fig:caseI1}
\end{figure}

\subsection{Case (II.4)}
In this experiment, the modulation parameters are chosen as 
$\gamma_J=0.65$, $\gamma_{\eps} = 0.35$, $\gamma_{\sigma}=0.35$.
The scalar internal data $Q$ is obtained by solving the forward problem~\eqref{eq:frequency}.
Then $1\%$ multiplicative noise is added to the signal.
The vector internal data $\vQ$ is found using the auxiliary solutions.
The reconstruction is performed according to Proposition~\ref{thm:case2}. That is, the boundary value problem~\eqref{eq:Eeq}, \eqref{eq:meas} is solved to obtain $\vE$. The Maxwell equations~\eqref{eq:frequency} are then used to find $\vJ$ through the similar approach of Case (I.1). The reconstructed source $\hat\bJ_{rec}$ is shown in Fig.~\ref{fig:caseII4}.

\begin{figure}[!htb]
    \centering
    \includegraphics[width=0.46\textwidth, height=0.38\textwidth]{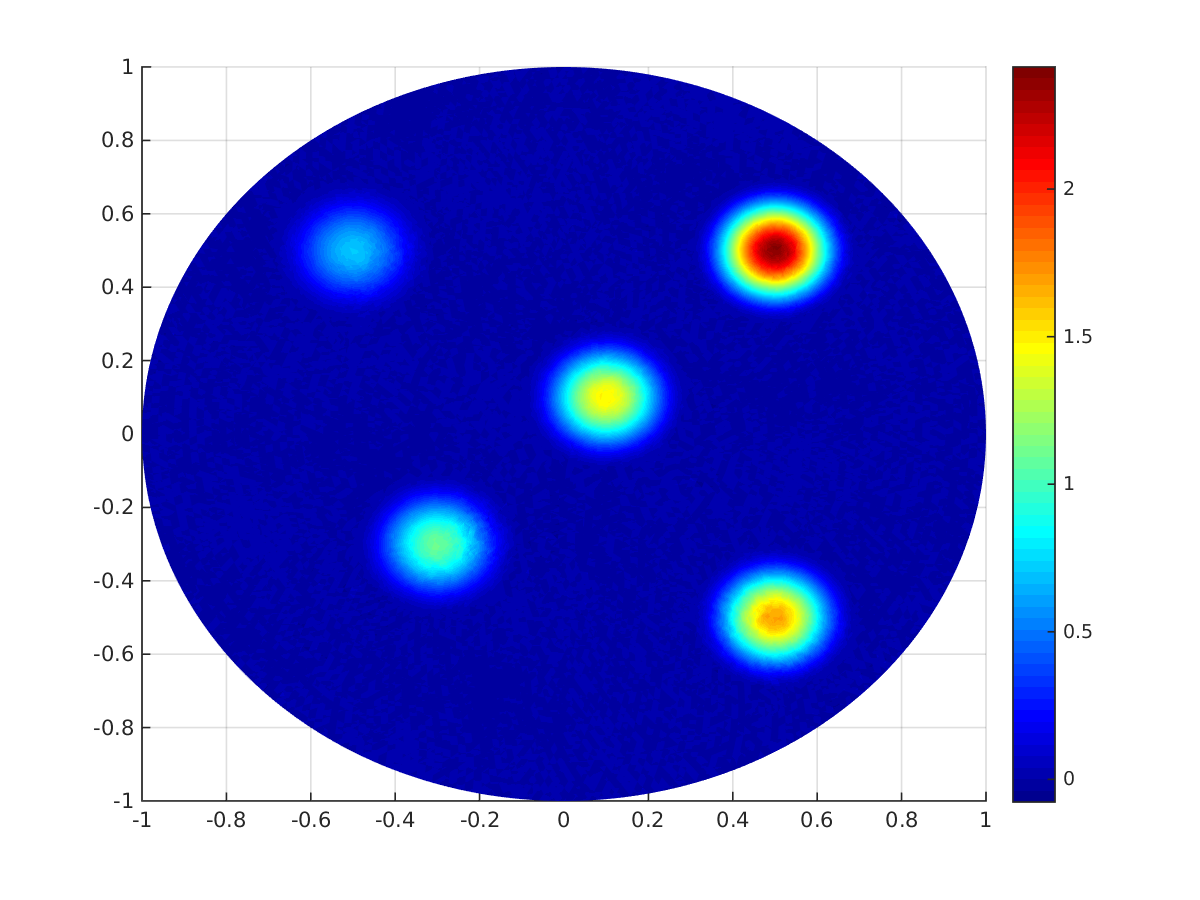}
    \includegraphics[width=0.46\textwidth, height=0.38\textwidth]{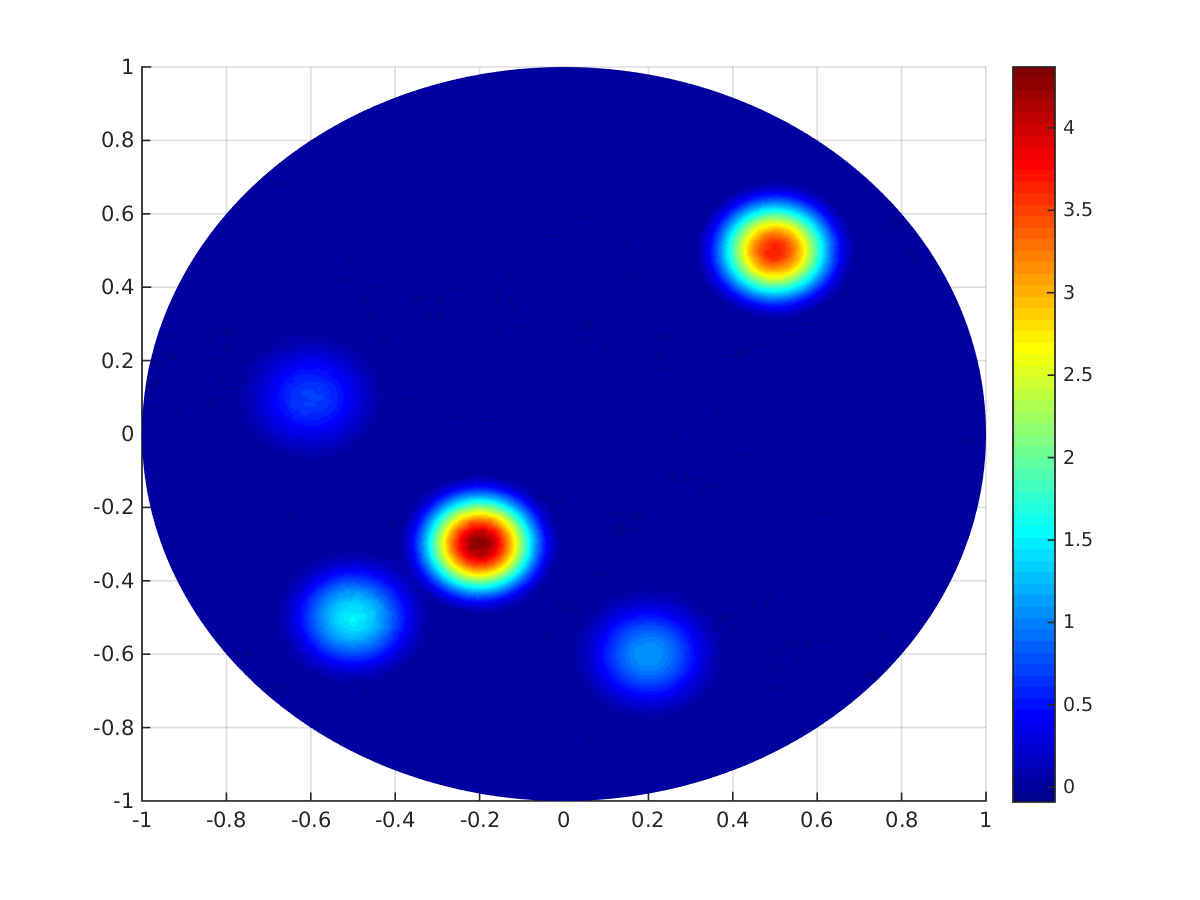}\\
    \includegraphics[width=0.46\textwidth, height=0.38\textwidth]{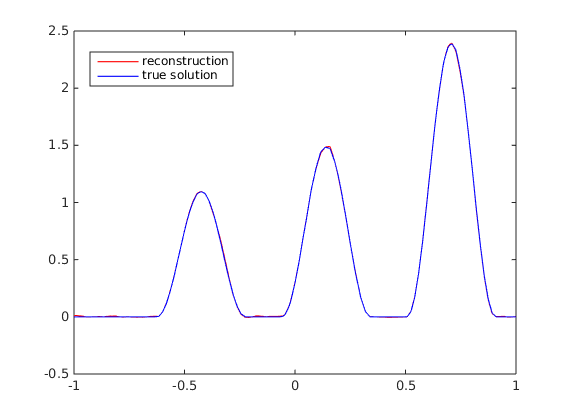}
    \includegraphics[width=0.46\textwidth, height=0.38\textwidth]{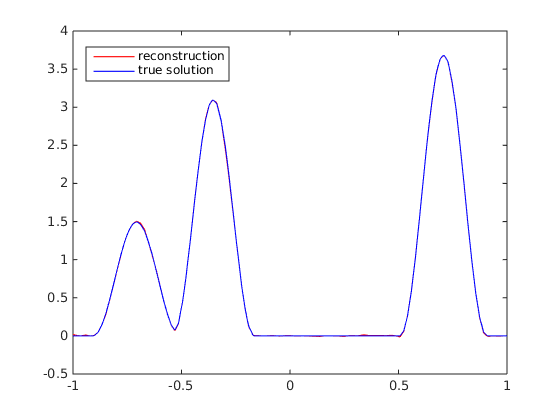}\\
    \caption{Reconstructed source current vector $\hat\bJ_{rec}$ for Case 6. Top row: from left to right, real part of $x$ component and $y$ component. Bottom row: from left to right, cross sections of reconstruction and the true solution along the line $y=x$. The relative $L^2$ error of reconstruction is $3.2\%$. }
    \label{fig:caseII4}
\end{figure}
\begin{remark}
The reconstructions for Case (II.4) are much better than those for Case (I.1). This can be explained by the corresponding stability estimates. Proposition~\ref{thm:case1} for Case (I.1) requires the $H(curl)$ norm of $\bQ$ to be bounded, which is very sensitive to noise. In contrast, the stability estimate in Proposition~\ref{thm:case2} for Case (II.4) only requires that the $L^2$ norm of $\vQ$ be bounded.
\end{remark}

\section{Acknowledgments}
The work of JCS was supported in part by the NSF grant DMS-1912821 and the AFOSR grant FA9550-19-1-0320. The work of YY is supported in part by the NSF grants DMS-1715178 and DMS-2006881.

\end{document}